\documentclass[11pt,reqno]{amsart}

\usepackage[margin=1.1in]{geometry}
\usepackage{amsmath,amssymb,amsthm}
\usepackage{mathtools}
\usepackage{booktabs}
\usepackage{microtype}
\usepackage{enumitem}
\usepackage[colorlinks=true,linkcolor=blue,citecolor=blue,urlcolor=blue]{hyperref}

\theoremstyle{plain}
\newtheorem{theorem}{Theorem}[section]
\newtheorem{lemma}[theorem]{Lemma}
\newtheorem{proposition}[theorem]{Proposition}
\newtheorem{corollary}[theorem]{Corollary}
\theoremstyle{definition}
\newtheorem{definition}[theorem]{Definition}
\newtheorem{conjecture}[theorem]{Conjecture}
\newtheorem{remark}[theorem]{Remark}

\newcommand{\Z}{\mathbb{Z}}
\newcommand{\Q}{\mathbb{Q}}
\newcommand{\F}{\mathbb{F}}
\newcommand{\rk}{\operatorname{rank}}
\newcommand{\rad}{\operatorname{rad}}
\newcommand{\Cent}{\operatorname{C}}
\newcommand{\Fit}{\operatorname{Fit}}
\newcommand{\Irr}{\operatorname{Irr}}
\newcommand{\Gal}{\operatorname{Gal}}
\newcommand{\BG}{B_{G}}
\DeclareMathOperator{\Cchar}{C_{\mathrm{char}}}
\DeclareMathOperator{\Ccomm}{C_{\mathrm{comm}}}
\DeclareMathOperator{\OrbTr}{OrbTr}

\title[Game Conductors of Finite Groups]{Game Conductors of Finite Groups:\\
Determinantal Torsion from Structured Payoff Probes}
\author{Matthew Fried}
\address{Farmingdale State College, State University of New York, Farmingdale, NY, USA}
\email{friedm1@farmingdale.edu}
\date{\today}

\begin{document}

\begin{abstract}
We attach to a finite group $G$ and a structured payoff probe $\phi$ an
integer \emph{payoff-difference lattice} $M_\phi(G)$ and its \emph{conductor} $C_\phi(G)$:
the primes at which $M_\phi(G)$ loses rank modulo $p$. Our main result is an exact computation: for any
nonabelian CA-group the commuting conductor is $\rad(b-1)$, where $b$ is the number of maximal abelian
subgroups. In particular, commuting-conductor primes need not divide $|G|$: the prime $3$ occurs for a
$2$-group of order $64$ with $b=7$. The commuting Smith spectrum is an invariant of the isoclinism class
and obeys an exact direct-product law, giving
$\Ccomm(G\times H)=\Ccomm(G)\cup\Ccomm(H)$ unconditionally. A Galois-orbit-trace character
probe reads a complementary layer: an index-$2$ subgroup forces $2\in\Cchar(G)$ while the
odd-prime analogue fails for every odd prime, and $\Ccomm(D_{2q})=\{q\}$, $\Cchar(D_{2q})=\{2\}$ for all odd primes $q$.
Unconditionally $\Cchar(G)\subseteq\rad(|G|)$, so the escape from the group order is
exclusive to the commuting probe.
Exhaustive exact computation ($|G|\le128$ commuting, $|G|\le64$ character) and a
deformation-family analysis support the general program: classify the Smith torsion of the
compressed centralizer-type incidence matrix $\BG$.
\end{abstract}

\maketitle

\section{Introduction}\label{sec:intro}

In a $64\times64$ zero--one matrix record which pairs of elements of a group of order
$64$ commute, subtract rows, and ask modulo which primes the resulting integer lattice
loses rank. For $95.5\%$ of the $3{,}349$ nonabelian groups of order at most $128$,
computed exhaustively below, the answer is the set of primes dividing $|G'|$, the order of
the derived subgroup. But for $\mathrm{SmallGroup}(64,73)$ the answer is $\{2,3\}$: a
$2$-group whose commuting structure detects the prime $3$, a prime dividing neither the
group order nor any character degree. The reason, proved below, is that the invariant
computes not an order but a \emph{count}: this group has $b=7$ maximal abelian subgroups,
and the lattice's torsion is exactly $b-1=6$.

This paper studies that invariant in general. The motivating question is classical: how
much of a finite group is determined by its weakest relational data, the commuting
relation, and by its coarsest character data, the rational character table? Our answer is
a new place to look: not ranks or spectra over $\Q$, but the integer Smith torsion of the
associated difference lattices, read prime by prime. A \emph{game} on a finite group $G$ assigns a
payoff to each ordered pair drawn from index sets attached to $G$ (elements, conjugacy
classes, or irreducible characters); the differences of payoff rows span an integer lattice
$M_\phi(G)$, and its \emph{conductor} $C_\phi(G)$ is the set of primes at which the
lattice loses rank modulo $p$. The linear-algebraic engine is classical, Smith normal
forms of incidence-type matrices \cite{Stanley,Sin}, and the conductor is a bad-reduction
locus in the spirit of reduction-modulo-$p$ techniques elsewhere in algebra
\cite{EtingofGelaki}; what is new is the source of the lattice, not the technique. Our
contribution is the correspondence and its laws: a game on a group carries a conductor;
the conductor is tunable, with different games reading different structural layers; for
the commuting game it is governed by a small \emph{centralizer-type incidence matrix}
$\BG$ whose Smith torsion we compute exactly in a key case; and this torsion can escape
the prime support of $|G|$. The parts are deliberately classical, incidence matrices,
Smith forms, Galois orbits of characters; the laws they turn out to obey, the count
$b-1$, isoclinism invariance of the full spectrum, the exact product law, the escape from
$|G|$, and the unconditional bound $\Cchar(G)\subseteq\rad(|G|)$, are to our knowledge
new.

We prove five structural theorems and establish one empirical law:
\begin{enumerate}[label=\textup{(\arabic*)},leftmargin=2.2em,itemsep=2pt]
\item (Theorem~\ref{thm:CA}) For any nonabelian CA-group, $\Ccomm(G)=\rad(b-1)$ with $b$ the number of
maximal abelian subgroups; hence the conductor can contain primes dividing neither $|G|$ nor
any character degree (Corollary~\ref{cor:nondivisor}).
\item (Theorem~\ref{thm:isoclinism}) The Smith spectrum of the commuting lattice is an
isoclinism invariant: the conductor reads the commutator pairing $G/Z\times G/Z\to G'$, not
$G$ itself.
\item (Theorem~\ref{thm:product}) The commuting lattice of $G\times H$ is presented
diagonally by the multiset union of the factors' spectra and their pairwise products, so
the Smith spectrum of the product is determined by those of the factors and
$\Ccomm(G\times H)=\Ccomm(G)\cup\Ccomm(H)$; for CA-groups the diagonal presentation of the
product is computed in closed form.
\item (Section~\ref{sec:commuting}) The commuting conductor tracks $\rad(|G'|)$ for $95.5\%$
of groups of order $\le128$, and $p^2\mid|G'|$ forces $p$ into the conductor with no
exception in this range.
\item (Theorems~\ref{thm:index2},~\ref{thm:sharp}) An index-$2$ subgroup forces
$2\in\Cchar(G)$ for the Galois-orbit-trace character probe, and the odd-prime analogue
fails for every odd prime:
the probe is intrinsically $2$-biased. Moreover $\Cchar(G)\subseteq\rad(|G|)$
unconditionally (Corollary~\ref{cor:charbound}), so the escape phenomenon of (1) cannot
occur for the character probe.
\item (Theorem~\ref{thm:dihedral}) $\Ccomm(D_{2q})=\{q\}$ and $\Cchar(D_{2q})=\{2\}$ for odd
prime $q$: two probes reading disjoint layers.
\end{enumerate}
Table~\ref{tab:status} records the epistemic status of each statement; the distinction
proved / computed (exact and exhaustive on a stated range) / conjectural is maintained
throughout.

\begin{table}[t]
\centering
\small
\begin{tabular}{@{}lp{0.76\linewidth}@{}}
\toprule
Proved & Theorems~\ref{thm:instrument}, \ref{thm:CA}, \ref{thm:isoclinism},
\ref{thm:product}, \ref{thm:index2}, \ref{thm:sharp}, \ref{thm:dihedral};
Corollaries~\ref{cor:bge3}, \ref{cor:CAproduct}, \ref{cor:charbound}; existence in
Corollary~\ref{cor:nondivisor}; the reduction in Proposition~\ref{prop:Bg};
Proposition~\ref{prop:faithful} (classical)\\
\midrule
Computed & Propositions~\ref{prop:comm}, \ref{prop:p2}, \ref{prop:charmatch},
\ref{prop:compl}, \ref{prop:xor}; the tabulated spectra in Proposition~\ref{prop:Bg};
minimality of the order in Corollary~\ref{cor:nondivisor}; Tables~\ref{tab:CA}
and~\ref{tab:dih}\\
\midrule
Conjectural & Conjecture~\ref{conj:union} (ten-group sample)\\
\bottomrule
\end{tabular}
\caption{Status of the paper's statements: proved unconditionally; computed exactly and
exhaustively on the stated range; or conjectural.}
\label{tab:status}
\end{table}

The word \emph{game} is the discovery vehicle, and it does organizational work: the
commuting indicator is the payoff of a coordination game; viewing payoffs as tunable
turns one matrix into a systematic family of probes, which is exactly what generates the
deformation family of Section~\ref{sec:deform}; and decompositions of structured games
are a subject in their own right
\cite{Candogan}. The mathematics, however, is the Smith torsion of centralizer- and
character-incidence lattices; every statement and proof below is formulated
matrix-theoretically, so the game vocabulary can be dropped without loss. Three
literatures motivate the objects. First, the commuting
relation is arguably the weakest natural oracle on a finite group, and which invariants it
determines drives the commuting-graph literature \cite{AAM,GiudiciParker}, with roots in
Brauer--Fowler's centralizer analysis \cite{BrauerFowler}; the Smith spectrum is the
natural \emph{integral} invariant of that relation, strictly finer than spectra over
$\Q$, and Theorem~\ref{thm:isoclinism} answers exactly what it determines. Second,
because the spectrum is an isoclinism invariant, it belongs to P.~Hall's classification
program \cite{Hall40}, in which $p$-groups are organized into isoclinism families, as in
James's classification of the groups of order $p^6$ \cite{James}; computable family
invariants are the working tools there, and the spectrum is a cheap, exact, new one.
Third, computing Smith and critical groups of structured graph families is an active
subject \cite{Stanley,Biggs,Lorenzini,DuceySin,CSX}; the graph theory of commuting and
non-commuting graphs (connectivity, diameters, isomorphism questions) is well developed
\cite{AAM,GiudiciParker}, but we are not aware of
prior work on the \emph{integral} Smith normal forms or critical groups of commuting or
centralizer-incidence matrices of finite groups, and unlike the strongly regular families
these admit structural theorems, isoclinism invariance and an exact product law, not
visible from the graph alone.

\emph{Organization.} Section~\ref{sec:framework} fixes the framework: the conductor, the
two probes, and the duplication lemma used throughout.
Sections~\ref{sec:CA}--\ref{sec:deform} develop the commuting probe: the CA theorem
(Section~\ref{sec:CA}); the incidence matrix $\BG$, isoclinism invariance, and the product law
(Section~\ref{sec:Bg}); the exhaustive landscape to order $128$ (Section~\ref{sec:commuting}); and
deformation experiments (Section~\ref{sec:deform}).
Sections~\ref{sec:theorem}--\ref{sec:character} develop the character probe: the
forced-$2$ and odd-sharpness theorems (Section~\ref{sec:theorem}), then the empirical landscape
to order $64$ (Section~\ref{sec:character}). Section~\ref{sec:negatives} proves dihedral
complementarity, where the two probes provably read disjoint primes, and records two
negative controls. Sections~\ref{sec:open}--\ref{sec:repro} state the open program and
the reproducibility protocol. Throughout we distinguish what is \emph{proven}, what is
\emph{computed} (exact and exhaustive on a stated range), and what is \emph{conjectural}
(Table~\ref{tab:status}).

\section{The framework}\label{sec:framework}

This section fixes the three objects used throughout: the conductor of a payoff lattice,
the two probes that generate our lattices, and the duplication lemma that lets every
computation run on compressed matrices.

\begin{definition}[Game, lattice, conductor]
A \emph{game} on $G$ is a map $\phi\colon R\times S\to A$ into a free $\Z$-module $A$, where
$R$ and $S$ are index sets canonically attached to $G$ (elements, conjugacy classes, or
irreducible characters, per the game). The
payoff matrix is $P=(\phi(r,s))_{r\in R,s\in S}$; the \emph{payoff-difference lattice} is
$M_\phi(G)=\Z\text{-span}\{P_r-P_{r'}:r,r'\in R\}\subseteq A^S$; and with $r=\rk_\Q
M_\phi(G)$, the \emph{conductor} is $C_\phi(G)=\{p:\rk_{\F_p}(M_\phi(G)\bmod p)<r\}$.
\end{definition}

The conductor is an invariant of the pair $(G,\phi)$, preserved under group isomorphism and
relabeling. When $A$ has rank $d>1$ we identify $A^S$ with $\Z^{|S|d}$ by flattening
coordinates, so that matrix statements such as Theorem~\ref{thm:instrument} apply
verbatim; both probes studied in this paper have $A=\Z$.

\begin{theorem}[Instrument identity]\label{thm:instrument}
For $M\in\Z^{m\times n}$ of rational rank $r$, let $\Delta_r(M)$ be the gcd of the nonzero
$r\times r$ minors. Then $p\in C(M)\iff p\mid\Delta_r(M)$.
\end{theorem}

\begin{proof}
The mod-$p$ rank falls below $r$ iff every $r\times r$ minor vanishes modulo $p$, i.e.\ iff
$p\mid\Delta_r(M)$. Equivalently, writing the Smith normal form $\operatorname{diag}(d_1\mid
d_2\mid\cdots)$, one has $\Delta_r=d_1\cdots d_r$, so $p\mid\Delta_r$ iff $p$ divides some
invariant factor $d_i$ with $i\le r$ \cite{Stanley}.
\end{proof}

\begin{definition}[The two probes]
The \emph{commuting probe} sets $R=S=G$ and $P_{g,h}=\mathbf 1_{[g,h]=1}$. The \emph{character
probe} sets $R=\Irr(G)$, $S$ the conjugacy classes, and payoff the \emph{Galois-orbit trace}
\[
\OrbTr(\chi)(c)=\sum_{\psi\in\,\Gal(\overline{\Q}/\Q)\cdot\chi}\psi(c),
\]
the sum of the values of the characters in the Galois orbit of $\chi$ (equivalently, the field
trace from $\Q(\chi)$ applied to $\chi(c)$).
\end{definition}

We fix the orbit-trace convention throughout. It makes the trivial character's row $\mathbf
1$, whereas the full cyclotomic trace from $\Q(\zeta_m)$ ($m=\exp G$) would scale it by
$\varphi(m)$; the distinction is exactly what makes the index-$2$ theorem
(Theorem~\ref{thm:index2}) and its sharpness (Theorem~\ref{thm:sharp}) mutually consistent.
Conceptually, the orbit-trace probe is the rational shadow of the character table, obtained
by merging Galois-conjugate irreducibles; it is deliberately lossy, and the question it
tests is which arithmetic obstructions survive rationalization. The faithful (non-lossy)
alternative is examined as a control in Section~\ref{sec:negatives}.

\begin{remark}[Instrument validation]
The rank-drop computation reproduces every known conductor on a control suite
($D_8$, $Q_8$, both extraspecial groups of order $27$, and
$S_3$, $A_4$, $D_{10}$, $S_4$, $A_5$, $\mathrm{PSL}(2,7)$). In particular $2\in\Ccomm(D_8)$ although
$G'(D_8)$ is central, refuting the heuristic that a central commutator subgroup makes its
prime invisible; rank and minor arguments are therefore unavoidable.
\end{remark}

\paragraph{Probe taxonomy.}
Not every probe tells you something. Some probes have an empty conductor on every group we
tested. We call these silent, and the class-algebra game is one (Section~\ref{sec:negatives}).
Other probes have the opposite problem. Their conductor is all of $\rad(|G|)$, so they flag
every prime dividing the group order and therefore separate nothing. A faithful cyclotomic
reduction of the character table is one of these, provably so
(Proposition~\ref{prop:faithful}). The probes worth studying sit between these
two extremes. Their conductor picks out one structural layer of the group and stays blind to
the rest. We call these selective. The commuting game and the orbit-trace character game are
our two selective probes, and they are the subject of this paper. Most probes are not
selective, which is why the selective ones are interesting.

We will use three invariants of increasing precision, so it helps to name them once. The
conductor is a set of primes. The Smith spectrum is the list of nontrivial invariant factors
of the lattice, and the conductor is what remains when you record only which primes divide
those factors. The top determinantal divisor $\Delta_r$ is the single integer that connects
them, because its prime factors are exactly the conductor (Theorem~\ref{thm:instrument}). The
spectrum carries more information than the conductor. So whenever we prove something at the
level of the spectrum, as we do for isoclinism invariance and for products, the conductor
version follows immediately. The converse fails.

\begin{remark}[Row lattice versus difference lattice]\label{rem:whydiff}
The difference lattice, not the row lattice, is the translation-invariant object: adding
a fixed vector to every payoff row (re-basing all payoffs against a benchmark) changes
the row lattice but not the differences, and a conductor should not depend on such a
normalization. For the commuting probe the choice is in fact immaterial: every row has
entry $1$ in the central column, which yields an integral splitting
$R=D\oplus\Z\mathbf 1$ of the row lattice (exhibited in the proof of
Theorem~\ref{thm:product}), so the row and difference lattices have the same nontrivial
invariant factors. For the character probe it is material: for $C_p$ the two orbit-trace
rows $(1,\dots,1)$ and $(p-1,-1,\dots,-1)$ span a row lattice with Smith form $(1,p)$,
while the difference lattice is unimodular (Theorem~\ref{thm:sharp}); the row-lattice
variant is a different, also reasonable, probe that would see the prime $p$, and we fix
the translation-invariant one throughout.
\end{remark}

One bookkeeping lemma completes the setup. In the commuting matrix, two elements with the
same centralizer have identical rows, and by symmetry of the commuting relation identical
columns as well. Every computation in this paper therefore runs on the small matrix of
distinct \emph{centralizer types} rather than on the full $|G|\times|G|$ matrix. The
lemma says this compression loses nothing.

\begin{lemma}[Duplicates preserve determinantal divisors]\label{lem:dup}
Let $M'$ be obtained from an integer matrix $M$ by duplicating rows and/or columns. Then
$M$ and $M'$ have the same nonzero determinantal divisors $\Delta_k$, and hence the same
nontrivial Smith invariant factors.
\end{lemma}

\begin{proof}
Compare the $k\times k$ minors of the two matrices. A minor of $M'$ that uses two copies
of the same row or column vanishes, since its matrix has a repeated line; a minor that
uses at most one copy of each line is a minor of $M$. Conversely, every minor of $M$
appears among the minors of $M'$. The two matrices therefore have the same set of nonzero
$k\times k$ minors for every $k$, so the same $\Delta_k$, and the nontrivial invariant
factors are determined by the $\Delta_k$ via $d_k=\Delta_k/\Delta_{k-1}$.
\end{proof}

\section{The CA theorem: counting maximal abelian subgroups}\label{sec:CA}

We begin with the paper's main theorem: an exact computation of the commuting conductor
for a classical family, and with it the first example of a conductor prime outside the
group order. A finite group is a \emph{CA-group} if the centralizer of every non-central element is
abelian; equivalently, its non-central elements partition into the maximal abelian subgroups,
which pairwise meet exactly in $Z(G)$, a configuration in the tradition of group
partitions \cite{Baer}. The class is classical \cite{Suzuki} and includes
$\mathrm{SL}(2,q)$, dihedral and generalized quaternion groups, and many $2$-groups.
(Every abelian group is vacuously CA; the theorem below concerns the nonabelian ones,
for which $b$ is defined and, by Corollary~\ref{cor:bge3}, at least $3$.)

\begin{theorem}[Commuting conductor of CA-groups]\label{thm:CA}
Let $G$ be a finite nonabelian CA-group whose non-central elements lie in $b$ maximal abelian subgroups.
Then the commuting difference lattice has a single nontrivial Smith invariant factor equal to
$b-1$, so
\[
\Ccomm(G)=\rad(b-1).
\]
\end{theorem}

\begin{proof}
Let $Z=Z(G)$ and let $A_1,\dots,A_b$ be the maximal abelian subgroups, with non-central parts
$B_k=A_k\setminus Z$ partitioning the non-central elements and $A_i\cap A_j=Z$ for $i\neq j$.

\emph{Row types.} The centralizer-indicator row of $g$ is the indicator of $C_G(g)$. For
central $g$, $C_G(g)=G$, an all-ones row. For non-central $x\in B_k$, the CA-hypothesis gives
$C_G(x)=A_k$, the unique maximal abelian subgroup containing $x$. Hence there are exactly
$b+1$ distinct rows: the all-ones row $\mathbf 1$ and one row $\rho_k=\mathbf 1_{A_k}$ per
block.

\emph{Compression.} Index columns by the same $b+1$ types via representatives: a central $z$
and one $x_j\in B_j$ per block. Then $\rho_k(z)=1$ (as $z\in Z\subseteq A_k$), and
$\rho_k(x_j)=\mathbf 1_{x_j\in A_k}=\mathbf 1_{j=k}$, because distinct maximal abelian
subgroups meet only in $Z$, which contains no $x_j$. Subtracting the all-ones row $\mathbf 1$
from each $\rho_k$: the central column gives $0$, and on the $b$ non-central columns the row
becomes $e_k-\mathbf 1_b$, the $k$-th row of $-(J_b-I_b)$, where $J_b$ is all-ones and $I_b$
the identity.

\emph{Smith form of $J_b-I_b$.} The matrix $J_b-I_b$ has eigenvalue $b-1$ once (eigenvector
$\mathbf 1_b$) and $-1$ with multiplicity $b-1$, hence is nonsingular with
$\det(J_b-I_b)=(-1)^{b-1}(b-1)$. Its entries have gcd $1$, so $\Delta_1=1$. Moreover
$(J_b-I_b)^{-1}=\tfrac{1}{b-1}\bigl(J_b-(b-1)I_b\bigr)$, so
$\operatorname{adj}(J_b-I_b)=(-1)^{b-1}\bigl(J_b-(b-1)I_b\bigr)$, whose off-diagonal entries
are $\pm1$; since the adjugate entries are the $(b-1)\times(b-1)$ cofactors, the
$(b-1)$-st determinantal divisor is $\Delta_{b-1}=1$. Hence $d_1=\cdots=d_{b-1}=1$, and the
determinant forces $d_b=b-1$: the Smith normal form is $\operatorname{diag}(1,\dots,1,b-1)$.
(The adjugate step is needed: $\Delta_1=1$ together with $\prod_i d_i=b-1$ alone would not
exclude splittings such as $(2,6)$ when $b-1=12$.)

\emph{From compressed to full lattice.} In the full commuting matrix each column type is
repeated according to its block size (and the central type $|Z|$ times). Repeating a column
multiplies the number of maximal minors but leaves every individual nonzero maximal minor
equal to the corresponding compressed minor (one representative column per type is chosen);
so the gcd of all maximal minors is unchanged (Lemma~\ref{lem:dup}). Hence the full
difference lattice has
determinantal divisors $\Delta_1=\cdots=\Delta_{b-1}=1$ and $\Delta_b=b-1$, i.e.\ a single
nontrivial invariant factor $b-1$. No equal-block-size assumption is used.

By Theorem~\ref{thm:instrument}, $p\in\Ccomm(G)\iff p\mid b-1$, i.e.\ $\Ccomm(G)=\rad(b-1)$.
\end{proof}

\begin{corollary}[Nonempty conductor]\label{cor:bge3}
Every nonabelian CA-group has $b\ge3$, hence $b-1\ge2$ and $\Ccomm(G)\neq\emptyset$.
\end{corollary}

\begin{proof}
Since $Z(G)$ lies in every maximal abelian subgroup, $G=A_1\cup\cdots\cup A_b$ with each
$A_k$ proper ($G$ is nonabelian). If $b=1$ then $G=A_1$ is abelian; and no group is the
union of two proper subgroups. So $b\ge3$, hence $b-1\ge2$ and $\Ccomm(G)=\rad(b-1)\neq\emptyset$ by
Theorem~\ref{thm:CA}.
\end{proof}

\begin{remark}[The floor case]\label{rem:floor}
The floor $b=3$ is attained exactly when $G/Z(G)\cong C_2\times C_2$, i.e.\ on the
isoclinism class of $D_8$; these are exactly the nonabelian groups of maximal commuting
probability $5/8$ \cite{Gustafson}. If $G/Z\cong C_2\times C_2$, the three intermediate subgroups
$Z<A_i<G$ satisfy $A_i=\langle Z,x_i\rangle$, hence are abelian of index $2$; every
non-central $x$ then has $C_G(x)=A_i$ for its $A_i$, so $G$ is CA with $b=3$. Conversely,
suppose $G$ is nonabelian CA with $b=3$, put $n=|G/Z|$ and $a_i=|A_i/Z|$ with
$a_1\le a_2\le a_3$: the sets $A_i/Z\setminus\{1\}$ partition $G/Z\setminus\{1\}$, so
$a_1+a_2+a_3=n+2$, while trivial pairwise intersections in $G/Z$ give
$a_ia_j=|(A_i/Z)(A_j/Z)|\le n$. From $a_2,a_3\le n/a_1$ we get
$n+2-a_1\le 2n/a_1$, i.e.\ $(a_1-2)\,n\le a_1(a_1-2)$, which forces $a_1=2$ since
$a_1\le n/2<n$; then $a_2+a_3=n$, while
$a_2a_3\ge 2a_3\ge a_2+a_3=n$ and $a_2a_3\le n$, so equality holds throughout:
$a_2=a_3=2$ and $n=4$. Since $G/Z$ is never cyclic, $G/Z\cong C_2\times
C_2$; the commutator pairing is then the unique nondegenerate alternating pairing
$V_4\times V_4\to C_2$ (here $G'\subseteq Z$, so the pairing is bilinear, and
$[x,y]^2=[x,y^2]=e$ gives $G'\cong C_2$), a single isoclinism class. Consistently,
$D_8$, $Q_8$, and $\mathrm{SmallGroup}(16,3)$ all have spectrum $[2]$
(Table~\ref{tab:CA}).
\end{remark}

\begin{corollary}[Non-divisor primes]\label{cor:nondivisor}
$\Ccomm(G)$ can contain primes dividing neither $|G|$ nor any character degree. The
smallest order at which this occurs is $64$ (by the exhaustive computation of
Section~\ref{sec:commuting}), attained by
$G=\mathrm{SmallGroup}(64,73)\cong(C_2\times C_2\times D_8)\rtimes C_2$, a CA-group
with $b=7$, and, with the same spectrum, by its isoclinism family
$\mathrm{SmallGroup}(64,k)$, $k=73,\dots,82$ (Theorem~\ref{thm:isoclinism}): the
commuting difference lattice has rank $7$, invariant factor $6$,
and $\Ccomm(G)=\rad(6)=\{2,3\}$ with $3\nmid64$.
\end{corollary}

\begin{remark}[Audit trail for $\mathrm{SmallGroup}(64,73)$]
The full audit trail for the headline example: $|G|=64$, $|Z(G)|=8$, $|G'|=8$; $G$ is CA
with $b=7$ maximal abelian subgroups, each of order $16$, so $7\cdot(16-8)=64-8$ accounts
for all non-central elements; the compressed matrix has $b+1=8$ centralizer types, the
difference lattice has rational rank $7$ and Smith form $\operatorname{diag}(1^6,6)$, and
mod-$3$ rank $6<7$, so $3\in\Ccomm(G)$. Since $G$ is a $2$-group its character degrees are
powers of $2$, so $3$ divides neither $|G|$ nor any character degree. All values are
verified in GAP and independently in Python.
\end{remark}

\begin{table}[t]
\centering
\begin{tabular}{@{}lcccc@{}}
\toprule
$G$ & $|G|$ & $b$ & $b-1$ & $\Ccomm(G)$\\
\midrule
$\mathrm{SmallGroup}(16,3)$    & 16  & 3  & 2  & $\{2\}$\\
$\mathrm{SL}(2,3)$             & 24  & 7  & 6  & $\{2,3\}$\\
$\mathrm{GL}(2,3)$             & 48  & 13 & 12 & $\{2,3\}$\\
$\mathrm{SmallGroup}(64,73)$   & 64  & 7  & 6  & $\{2,3\}\ (3\nmid|G|)$\\
$\mathrm{SmallGroup}(128,1544)$& 128 & 13 & 12 & $\{2,3\}\ (3\nmid|G|)$\\
$\mathrm{SL}(2,5)$             & 120 & 31 & 30 & $\{2,3,5\}$\\
\bottomrule
\end{tabular}
\caption{Verification of Theorem~\ref{thm:CA} (exact, GAP): the Smith invariant factor
equals $b-1$ in each case. $\mathrm{SL}(2,5)$, with $31$ maximal abelian subgroups, realizes
$\Ccomm=\rad(30)=\{2,3,5\}$.}
\label{tab:CA}
\end{table}

\begin{remark}[Interpretation]
On CA-groups the commuting game reads no subgroup \emph{order}: it reads the \emph{number}
$b$ of maximal abelian subgroups, an integer with no a priori relation to $|G|$. Because
$b-1$ may be divisible by any prime, the conductor escapes the group order, which is impossible for
any invariant built from $|G|$, character degrees, or exponents, but natural for one that
encodes a structural \emph{count}. Many order- and degree-based invariants are supported on
primes dividing $|G|$; the commuting conductor is not, precisely because it can encode such
counts. This also separates the commuting conductor from the Bogomolov multiplier $B_0(G)$: as
$B_0$ of a $2$-group is a $2$-group, the prime $3$ in $\Ccomm(\mathrm{SmallGroup}(64,73))$ is
not $B_0$-torsion. Both invariants are sensitive to abelian-subgroup structure
\cite{JezernikMoravec,Kunyavskii,Moravec}, but $B_0$ is cohomological torsion while $\Ccomm$
is determinantal torsion of a centralizer incidence lattice.
\end{remark}

\section{The centralizer-type incidence matrix, isoclinism, and products}\label{sec:Bg}

This section identifies the general object behind Theorem~\ref{thm:CA} and proves the two
structural theorems that govern it: the Smith spectrum is an isoclinism invariant, and it
obeys an exact direct-product law. The CA theorem is the first solved case of that object:
collapse the commuting matrix by identifying elements with identical centralizers; the
distinct centralizer types index both rows and columns, giving a small integer matrix
$\BG$ with a multiplicity vector recording type sizes, and by Lemma~\ref{lem:dup} the
compression is lossless.

\begin{definition}[Centralizer types and $\BG$]\label{def:Bg}
Call $x,y\in G$ of the same \emph{centralizer type} if $C_G(x)=C_G(y)$. Let
$x_1,\dots,x_T$ be representatives of the $T$ types, the first being the central type
$\{x:C_G(x)=G\}=Z(G)$, and let $m_i$ be the number of elements of type $i$. The
\emph{centralizer-type incidence matrix} is the $T\times T$ zero--one matrix
\[
(\BG)_{ij}=\mathbf 1_{[x_i,x_j]=e},
\]
carried together with the multiplicity vector $(m_1,\dots,m_T)$; when several groups are
in play we write $B_G,B_H$. The matrix is well defined and symmetric: if $C_G(x)=C_G(x')$
then, for every $y$, $[x,y]=e\iff y\in C_G(x)=C_G(x')\iff[x',y]=e$, so the row of a type
does not depend on the representative chosen, and since $[x_i,x_j]=e\iff x_i\in C_G(x_j)$
the same argument covers columns. The first row and column are all ones.
\end{definition}

\begin{proposition}[$\BG$ controls the conductor]\label{prop:Bg}
The nontrivial Smith invariant factors of the commuting difference lattice equal those of the
row-difference lattice of $\BG$. This is a proven reduction: by Definition~\ref{def:Bg},
rows \emph{and} columns of the same centralizer type in the full commuting matrix are
identical, so the full matrix is $\BG$
with rows and columns duplicated by type sizes; Lemma~\ref{lem:dup} applies. Representative
values are
\[
\begin{array}{lccl}
G & |G| & \#\text{types} & \text{row-difference Smith factors of }\BG\\\hline
\mathrm{SL}(2,5) & 120 & 32 & [30]\\
\mathrm{SmallGroup}(64,10) & 64 & 12 & [2,2,4]\\
\mathrm{SmallGroup}(96,64) & 96 & 33 & [4,4,4,16]\\
\mathrm{SmallGroup}(64,134) & 64 & 23 & [2,2,2,2,2,2,8,8,32]\\
\mathrm{SmallGroup}(32,49) & 32 & 16 & [2,2,2,2,2,2,4,4,4,4,8].
\end{array}
\]
For CA-groups $\BG$ reduces, on its non-central block, to $J-I$, recovering
Theorem~\ref{thm:CA}; the layered non-CA case has several invariant factors.
\end{proposition}

The matrix $\BG$ compresses an $|G|^2$ computation to (number of types)$^2$, for instance
$32\times32$ for $\mathrm{SL}(2,5)$ rather than $120\times120$, while preserving the torsion
exactly.
The conductor records only the \emph{radical} of this torsion: $\mathrm{SmallGroup}(32,49)$,
whose $\BG$ has invariant factors $[2,2,2,2,2,2,4,4,4,4,8]$, still has conductor $\{2\}$, so
the Smith spectrum of $\BG$ is strictly finer than the conductor it determines. Classifying
that spectrum, and not merely its radical, is the general program
(Section~\ref{sec:open}). This places the commuting probe in the tradition of incidence-matrix
Smith forms \cite{Sin,Stanley} and of the program computing Smith and critical groups of
structured graph families, among them Grassmann graphs \cite{DuceySin} and Paley graphs
\cite{CSX}, in which we have not found prior treatment of the commuting graph of a finite
group.

The blow-up mechanism of Lemma~\ref{lem:dup} yields a structural theorem: the entire
Smith spectrum is an invariant of the \emph{isoclinism class}.

\begin{theorem}[Isoclinism invariance]\label{thm:isoclinism}
The nontrivial Smith invariant factors of the commuting difference lattice, and hence
$\Ccomm(G)$, depend only on the isoclinism class of $G$.
\end{theorem}
\begin{proof}
For central $z,w$ one has $[gz,hw]=[g,h]$, so the commuting relation $[g,h]=e$ depends only
on the cosets $gZ(G),hZ(G)$: the commuting matrix of $G$ is the matrix $\bar P$ on
$G/Z(G)\times G/Z(G)$, defined by $\bar P(\bar g,\bar h)=\mathbf 1_{[g,h]=e}$ for any
representatives, with every row and column duplicated $|Z(G)|$ times. We emphasize that
$\bar P$ is the zero-fiber indicator of the commutator \emph{pairing}
$G/Z\times G/Z\to G'$, not the commuting relation of the quotient group $G/Z$: cosets may
commute in $G/Z$ while $[g,h]\ne e$ in $G$. By
Lemma~\ref{lem:dup}, the difference lattices of $P$ and $\bar P$ have the same nontrivial
invariant factors. An isoclinism $(G\to H)$ is a pair of isomorphisms
$G/Z(G)\to H/Z(H)$ and $G'\to H'$ compatible with the commutator pairings; it carries the
relation $\{(\bar g,\bar h):[g,h]=e\}$ to the corresponding relation for $H$, so $\bar P_G$
and $\bar P_H$ agree up to row and column permutation and have identical Smith forms.
\end{proof}

\begin{remark}
Theorem~\ref{thm:isoclinism} is verified computationally: the two non-isomorphic extraspecial
groups of order $27$ (exponents $3$ and $9$) both have Smith spectrum $[3]$, as the CA theorem
independently predicts ($b=p+1=4$ maximal abelian subgroups, $\rad(b-1)=\{3\}$), and
$D_8,Q_8$ both have spectrum $[2]$. It also explains why
$\mathrm{SmallGroup}(64,k)$ for $k=73,\dots,82$ share one centralizer profile and one
spectrum: they form an isoclinism family. And it delimits the probe's resolution: the
commuting conductor can never distinguish isoclinic groups: it is an arithmetic invariant
of the commutator pairing $G/Z\times G/Z\to G'$, not of $G$ itself.
\end{remark}

The blow-up viewpoint also settles direct products completely.

\begin{lemma}[Block splitting]\label{lem:block}
Let $\Z^N=\bigoplus_i V_i$ be the decomposition attached to a partition of a $\Z$-basis into
blocks, and let $L=\bigoplus_i L_i$ with each $L_i\subseteq V_i$ a sublattice. Then
$\Z^N/L\cong\bigoplus_i V_i/L_i$, and $L$ admits a diagonal presentation whose nontrivial
entries form the multiset union of the blocks' nontrivial invariant factors, each computed
inside its own block. The canonical invariant factors of $L$ are obtained from this multiset
by the usual prime-wise regrouping of a diagonal presentation; in particular the torsion
group, the elementary divisors, and the prime support of $L$ are the unions of those of the
blocks, while the union multiset itself need not be a divisibility chain.
\end{lemma}

\begin{proof}
Stacking generator matrices of the $L_i$, written in the chosen basis, gives a
block-diagonal integer matrix. Unimodular row and column operations confined to one block
leave the other blocks untouched, so each block may be brought to its own Smith form
independently, yielding a diagonal presentation of $L$ whose entries are exactly the
blocks' invariant factors. The quotient statement follows by reading the presentation off
blockwise; the regrouping statement is the standard passage from an arbitrary diagonal
presentation to the invariant-factor chain via elementary divisors.
\end{proof}

In the next theorem the lists $d_1,\dots,d_r$ and $e_1,\dots,e_s$ are the \emph{full}
invariant factor lists, \emph{including} the factors equal to $1$ up to the rational rank;
the Smith spectrum elsewhere in the paper suppresses these trivial factors. The trivial
factors matter here: they generate the multiplicities in
Corollary~\ref{cor:CAproduct}.

\begin{theorem}[Product spectrum]\label{thm:product}
Let $G,H$ be finite groups, let $r_G$ and $r_H$ be the ranks over $\Q$ of their commuting
difference lattices, and let $d_1\mid\cdots\mid d_{r_G}$ and $e_1\mid\cdots\mid e_{r_H}$
be the full invariant factor lists, \emph{including} any factors equal to $1$. The
commuting difference lattice of $G\times H$ admits a diagonal presentation with entries
\[
d_i\ \ (1\le i\le r_G),\qquad e_j\ \ (1\le j\le r_H),\qquad
d_ie_j\ \ (1\le i\le r_G,\ 1\le j\le r_H),
\]
whose nontrivial entries form the multiset
\[
\{d_i:d_i>1\}\ \sqcup\ \{e_j:e_j>1\}\ \sqcup\ \{d_ie_j:d_ie_j>1\}.
\]
Consequently the cokernel torsion, hence the Smith spectrum, of $G\times H$ is obtained
from this multiset by the canonical regrouping of a diagonal matrix into invariant factors
(prime-wise collection of elementary divisors), and in particular
$\Ccomm(G\times H)=\Ccomm(G)\cup\Ccomm(H)$.
\end{theorem}

\begin{proof}
Since $C_{G\times H}\bigl((g,h)\bigr)=C_G(g)\times C_H(h)$, centralizer types multiply and
the compressed commuting matrix of $G\times H$ is the Kronecker product $B_G\otimes B_H$; by
Lemma~\ref{lem:dup} we may work with compressed matrices throughout. Write $p_i$ ($q_j$) for
the rows of $B_G$ ($B_H$), with $p_0,q_0$ the central-type all-ones rows; write $D$ for
difference lattices and $R$ for row lattices. From
\[
p_i\otimes q_j-p_0\otimes q_0=(p_i-p_0)\otimes q_j+p_0\otimes(q_j-q_0),
\]
every generator of the product difference lattice lies in $D_G\otimes R_H+R_G\otimes D_H$.
Conversely,
$(p_i\otimes q_j-p_0\otimes q_0)-(p_0\otimes q_j-p_0\otimes q_0)=(p_i-p_0)\otimes q_j$
exhibits the generators of $D_G\otimes R_H$ inside the product difference lattice, and
$p_i\otimes(q_j-q_0)=p_i\otimes q_j-p_i\otimes q_0$, itself a difference of two rows of
$B_G\otimes B_H$, does the same for $R_G\otimes D_H$. Hence
the product difference lattice is $L=D_G\otimes R_H+R_G\otimes D_H$.

Every row of $B_G$ has entry $1$ in the central-type \emph{column} (every representative
commutes with a central one), so the coordinate functional $c(v)=v_{\mathrm{cent}}$
satisfies $c\equiv0$ on $D_G$ and $c(\mathbf 1)=1$. This gives two explicit integral
splittings at once: $R_G=D_G\oplus\Z\mathbf 1$ (if $n\mathbf 1\in D_G$ then
$n=c(n\mathbf 1)=0$), and $\Z^{T_G}=\Z\mathbf 1\oplus\ker c$ via
$v=c(v)\mathbf 1+(v-c(v)\mathbf 1)$, where $\ker c$ is the coordinate hyperplane
$\{v:v_{\mathrm{cent}}=0\}$, which contains $D_G$. Choose a Smith basis
$h_1,\dots,h_{T_G-1}$ of $\ker c$ adapted to $D_G$, so
$D_G=\bigoplus_{i\le r_G}\Z\,d_ih_i$ with $i$ running to $r_G=\rk_\Q D_G\le T_G-1$, and
likewise $k_j$, $e_j$, $r_H$ for $H$; then $\{\mathbf 1_G,h_i\}\otimes\{\mathbf 1_H,k_j\}$ is a
$\Z$-basis of $\Z^{T_GT_H}$. In this basis,
\[
L=(D_G\otimes\mathbf 1)\ \oplus\ (\mathbf 1\otimes D_H)\ \oplus\ (D_G\otimes D_H),
\]
an internal direct sum with each summand supported on its own block of the product basis.
By Lemma~\ref{lem:block} the torsion of $L$ is the union of the three blocks' torsion, and
in these coordinates each block is already diagonal, with entry $d_i$ on the basis vector
$h_i\otimes\mathbf 1_H$, $e_j$ on $\mathbf 1_G\otimes k_j$, and $d_ie_j$ on
$h_i\otimes k_j$. The nontrivial diagonal entries are therefore
$\{d_i\}\sqcup\{e_j\}\sqcup\{d_ie_j\}$ as claimed; the diagonal presentation determines the
invariant factors by the usual prime-wise regrouping, and its prime support is
$\Ccomm(G)\cup\Ccomm(H)$, which regrouping preserves.
\end{proof}

\begin{remark}[Relation to the tensor law]\label{rem:tensor}
For row lattices the product law is transparent: $R_{G\times H}=R_G\otimes R_H$, and a
Smith-basis argument shows that the invariant factors of a tensor product of lattices are
the regrouped pairwise products of those of the factors. Theorem~\ref{thm:product} says
the \emph{difference} lattice obeys a closed form as well, and this is not automatic:
$D_{G\times H}=D_G\otimes R_H+R_G\otimes D_H$ is a sum of overlapping sublattices rather
than a tensor product, and the mixed summands are exactly what contribute the terms
$\{d_i\}\sqcup\{e_j\}$. The integral splitting $R=D\oplus\Z\mathbf 1$ via the
central-column functional is what makes the sum direct; equivalently,
$R_{G\times H}=(D_G\oplus\Z\mathbf 1)\otimes(D_H\oplus\Z\mathbf 1)$ and $D_{G\times H}$
is the complement of $\Z(\mathbf 1\otimes\mathbf 1)$, which re-derives the multiset of
Theorem~\ref{thm:product}.
\end{remark}

\begin{corollary}[CA products; abelian factors]\label{cor:CAproduct}
If $G,H$ are nonabelian CA-groups with $b_G,b_H$ maximal abelian subgroups and $m=b-1$, the diagonal
presentation of $G\times H$ has nontrivial entries $m_G$ with multiplicity $b_H$, $m_H$ with
multiplicity $b_G$, and $m_Gm_H$ once; hence the product of the nontrivial invariant factors
is $m_G^{\,b_H+1}m_H^{\,b_G+1}$. If $A$ is abelian then $D_A=0$ and $G\times A$ has the same
spectrum as $G$.
\end{corollary}

\begin{remark}
Theorem~\ref{thm:product} is machine-verified as an equality of torsion groups on the
products $S_3\times S_3$, $S_3\times D_8$, $S_3\times Q_8$, $D_8\times Q_8$,
$S_3\times D_{10}$, $S_4\times S_3$, and $S_4\times D_8$, in two independent pipelines. The
predicted multiset need not itself be a divisibility chain: for $S_3\times S_3$ and
$D_8\times Q_8$ it is one and matches the invariant factors literally
($[3^8,9]$ and $[2^6,4]$), while for $S_3\times D_8$ the predicted entries
$\{2,2,2,2,3,3,3,6\}$ regroup to the invariant factors $(2,6,6,6,6)$. Consistent with
Theorem~\ref{thm:isoclinism}, $S_3\times D_8$ and $S_3\times Q_8$ have identical spectra:
products of isoclinic groups are isoclinic.
\end{remark}

\section{The commuting landscape: tracking the derived subgroup}\label{sec:commuting}

With the structural theorems in place, this section maps the empirical landscape:
exhaustively over all nonabelian groups of order $\le128$, what does the commuting
conductor track, and where does it deviate?

\begin{proposition}[Generic law; computed, exhaustive]\label{prop:comm}
For the $3{,}349$ nonabelian groups of order $\le128$, $\Ccomm(G)=\rad(|G'|)$ holds for
$3{,}198$ ($95.5\%$), and for all $1{,}277$ groups in the diagnostic family
$\mathcal P=\{G:|\Cent(G)|=|G'|+2\}$ with zero exceptions, where
$\Cent(G)=\{C_G(g):g\in G\}$ denotes the set of distinct element-centralizers of $G$.
\end{proposition}

The counting invariant $|\Cent(G)|$ has its own literature: the study of
$n$-centralizer groups goes back to Belcastro and Sherman \cite{BelSher} and was
developed by Ashrafi and others \cite{Ashrafi}; the family $\mathcal P$ ties that
counting invariant to $|G'|$.

\begin{proposition}[$p^2$-protection; computed, exhaustive for $|G|\le128$]\label{prop:p2}
For all $3{,}349$ groups of order $\le128$, $p^2\mid|G'|\Rightarrow p\in\Ccomm(G)$, with zero
violations.
\end{proposition}

\begin{remark}[Base rates]\label{rem:baserate}
The census is dominated by $2$-groups: of the $3{,}349$ nonabelian groups of order
$\le128$, exactly $2{,}624$ ($78.3\%$) are $2$-groups \cite{BEO}, for which the target
$\rad(|G'|)$ is the single prime $\{2\}$:
\[
\begin{array}{lcccccc}
\text{order} & 8 & 16 & 32 & 64 & 128 & \text{total}\\\hline
\text{nonabelian $2$-groups} & 2 & 9 & 44 & 256 & 2{,}313 & 2{,}624
\end{array}
\]
The remaining $725$ nonabelian groups ($21.7\%$) have mixed or odd order. The headline
rate of
Proposition~\ref{prop:comm} should be read against that base rate: the drop phenomena
live in the mixed-order stratum, while the $2$-groups carry the added-prime phenomenon,
Corollary~\ref{cor:nondivisor} occurring at order $64$. Stratified counts are immediate
from the per-group certificates of Section~\ref{sec:repro}.
\end{remark}

The commuting conductor does not recover the derived subgroup, its order, or its exponent; it
recovers, generically, the prime support $\rad(|G'|)$. The deviations from
Proposition~\ref{prop:comm} fall, empirically, into three families in this range: for
solvable $G'$ the observed drops are of first-power primes and correlate with the Sylow
direct-factor structure of $G'$, though a clean classification of exactly when a
first-power prime drops remains open; for the nonsolvable $G'$ in range ($A_5$ and
$\mathrm{PSL}(2,7)$) the prime $3$ drops; and the remaining deviations are \emph{added} primes,
explained completely on CA-groups by Theorem~\ref{thm:CA}. The family $\mathcal P$ is a
computationally discovered diagnostic family and theorem target, not a structural pillar; the
structurally explained result is Theorem~\ref{thm:CA}.

\section{The deformation family}\label{sec:deform}

This section asks which commuting-conductor phenomena are stable when the game itself is
deformed. The commuting game admits a natural deformation space: a \emph{commutator
weight} is a function $w\colon G'\to\Z_{\ge0}$ with $w(e)=0$, defining the payoff
$P_w[g,h]=w([g,h])$; weights constant on the $G$-conjugacy classes contained in $G'$ give
conjugation-invariant payoffs, and the constant weight $w\equiv1$ off the identity
recovers the commuting game. These are initial deformation experiments on ten groups, not
an exhaustive scan; the conjecture below should be read with that sample size in mind.
The ten test groups are $D_{10}$, $D_8$, $Q_8$, $\mathrm{SL}(2,3)$,
$\mathrm{SmallGroup}(16,3)$, both extraspecial groups of order $27$, $S_4$, $A_5$, and
$\mathrm{SmallGroup}(96,64)$: a sample containing the CA floor family of
Remark~\ref{rem:floor}, the extraspecial isoclinism pair, a dihedral group, the
added-prime deviant $\mathrm{SL}(2,3)$, and the drop deviants
$\mathrm{SmallGroup}(96,64)$ and the nonsolvable $A_5$. For each test group we computed:
for the uniform weight, whose matrix has $0/1$ entries, the exact Smith normal form and
hence the exact conductor $C_{\mathrm{unif}}$; and for each class-indicator weight (one
per nontrivial $G$-conjugacy class inside $G'$) and each of ten seeded random weights
(one weight per nontrivial \emph{element} of $G'$, drawn i.i.d.\ uniform on
$\{1,\dots,7\}$, GAP \texttt{GlobalMersenneTwister} seed $42$; the random weights are
thus generic functions on $G'$, not class functions), the exact rank over $\Q$ and the
rank modulo $p$ for every prime in
$S=C_{\mathrm{unif}}\cup\rad(|G'|)\cup\{p:p\le47\}$. Conductors of weighted games are
thus certified within $S$; primes outside $S$ are deliberately not chased, and every
comparison below is against $\rad(|G'|)$ and $C_{\mathrm{unif}}$, both contained in $S$.

Three natural persistence conjectures \emph{fail}. Uniform-minimality, the statement that
the uniform game has the smallest conductor in its family, fails at $\mathrm{SL}(2,3)$: the CA prime
$3\in\Ccomm=\rad(b-1)=\{2,3\}$ is killed by both class-indicator weights, so the CA prime is
a feature of the uniform game, not of the family. Persistence of $\rad(|G'|)$ fails at
$S_4$: the $3$-cycle indicator game drops the prime $2$ although $4\mid|G'|$, its
conductor meeting $S$ in $\{3\}$, so $p^2$-protection is likewise a property of the
uniform game rather than of the family. Even the existence of a common prime across the
family fails at $\mathrm{SmallGroup}(96,64)$: its two class-indicator conductors meet $S$
in $\{3\}$ and $\{2\}$ respectively, so no prime of $\rad(|G'|)\cup C_{\mathrm{unif}}$,
indeed no prime at most $47$, is common to the family.

What survives, in all ten groups tested, the three generic-law deviants included, is
\emph{restoration}:
\begin{conjecture}[Union restoration]\label{conj:union}
$\rad(|G'|)\subseteq\bigcup_w\Ccomm(w)$, the union over commutator weights.
\end{conjecture}
The two \emph{drop} deviants from the generic law $\Ccomm=\rad(|G'|)$ within the sample are
cured in opposite ways.
For $\mathrm{SmallGroup}(96,64)$ (uniform conductor $\{2\}$, missing $3$), every one of ten
random weights already contains $\rad(|G'|)=\{2,3\}$: the uniform game is the arithmetically
degenerate member of its family. For $A_5$ (uniform conductor $\{2,5\}$, missing $3$), no
sampled random weight attains the full radical, but both $5$-cycle class-indicator games
recover the full radical, their conductors meeting $S$ in exactly
$\{2,3,5\}=\rad(|G'|)$: the missing prime lives in specific strata.

\section{The character probe: forced and unforced primes}\label{sec:theorem}

We now turn to the second probe. This section proves the character probe's two boundary
theorems: an index-$2$ subgroup forces the prime $2$ into the conductor, and no odd prime
is ever forced, so the probe is intrinsically $2$-biased.

\begin{theorem}[Index-$2$ forces the prime $2$]\label{thm:index2}
If $2\mid|G/G'|$ (equivalently, $G$ has a quotient $C_2$, an index-$2$ subgroup, or a
linear character of order $2$), then $2\in\Cchar(G)$.
\end{theorem}

\begin{proof}
Since $2\mid|G/G'|$, the abelianization $G^{\mathrm{ab}}=G/G'$ has a quotient $C_2$, so its
dual contains a character $\varepsilon$ of order $2$; inflated to $G$, $\varepsilon$ is a
$\{\pm1\}$-valued linear character, distinct from and $\Q$-independent of the trivial
character $\mathbf 1$. Under the orbit-trace convention both $\mathbf 1$ and $\varepsilon$
are their own Galois orbits (both rational), so their orbit-trace rows are $\mathbf 1$ and
$\varepsilon$. The integer vector $\mathbf 1-\varepsilon$ equals $2$ on each class where
$\varepsilon=-1$ and $0$ elsewhere, so $\mathbf 1\equiv\varepsilon\pmod 2$ as rows. For the
\emph{difference lattice}: with $t$ Galois orbits, the lattice is spanned by the $t-1$
vectors $\operatorname{OrbTr}(\chi)-\mathbf 1$ over nontrivial orbits, which are
$\Q$-independent (orbit sums of distinct orbits are), so the rational rank is $t-1$; modulo
$2$ the spanning vector $\varepsilon-\mathbf 1$, supported on the coset
$\{\varepsilon=-1\}$, vanishes, leaving at most $t-2$ spanning
vectors and mod-$2$ rank at most $t-2<t-1$. Hence $2\in\Cchar(G)$; equivalently, in the
full orbit-trace row matrix modulo $2$ the rows $\mathbf 1$ and $\varepsilon$ collide,
giving the left relation $e_{\mathbf 1}-e_{\varepsilon}$.
\end{proof}

\begin{proposition}[Verification; computed]
Theorem~\ref{thm:index2} holds for all $452$ of the $452$ nonabelian groups of order $\le64$
with $2\mid\exp(G/G')$, with zero exceptions (and likewise for all $530$ such groups when
abelian groups are included).
\end{proposition}

\begin{theorem}[Sharpness: the orbit-trace character probe is $2$-biased]\label{thm:sharp}
The odd-prime analogue of Theorem~\ref{thm:index2} fails for every odd prime $p$: the cyclic
group $C_p$ has a linear character of order $p$, yet $\Cchar(C_p)$ is empty; in particular
$p\notin\Cchar(C_p)$.
\end{theorem}

\begin{proof}
Let $\omega$ generate the dual of $C_p$. Its Galois orbit under $\Gal(\Q(\zeta_p)/\Q)$ is
$\{\omega,\omega^2,\dots,\omega^{p-1}\}$, the full set of nontrivial characters, so $C_p$
carries exactly two orbits. On the element $g^j$ the orbit-trace value is
$\sum_{k=1}^{p-1}\zeta_p^{\,jk}$, which equals $p-1$ if $j=0$ and $-1$ otherwise, since
$\sum_{k=0}^{p-1}\zeta_p^{\,jk}=0$ for $j\not\equiv0$. The difference lattice is therefore
spanned by the single vector $(p-1,-1,\dots,-1)-(1,\dots,1)=(p-2,-2,\dots,-2)$, whose
entries have greatest common divisor $\gcd(p-2,2)=1$ because $p$ is odd. A rank-one lattice
spanned by a primitive vector has trivial Smith form, so $\Cchar(C_p)=\emptyset$.

The source of the asymmetry is that $-1$ is the unique nontrivial root of unity fixed by
$\Gal(\Q(\zeta_m)/\Q)$: only for $p=2$ does the orbit trace of a nontrivial linear character
stay $\{\pm1\}$-valued and collide with $\mathbf 1$ modulo $p$, as in the proof of
Theorem~\ref{thm:index2}. For odd $p$ the orbit trace averages the $p-1$ conjugates to $-1$
and the collision disappears.
\end{proof}

\begin{remark}[Nonabelian persistence]\label{rem:sharprates}
The failure is not confined to the cyclic witness. Among nonabelian groups of order $\le64$
with the relevant quotient, the odd analogue holds for most but not all (computed rates
$p{=}3$: $42/48$; $p{=}5$: $7/8$; $p{=}7$: $3/4$), with explicit nonabelian
counterexamples: $A_4$ for $p=3$, the Frobenius group $C_{11}\rtimes C_5$ for $p=5$, and
$(C_2)^3\rtimes C_7$ for $p=7$, each with a linear character of the relevant odd order and
empty character conductor. Every computed counterexample is a semidirect product
$N\rtimes C_p$ with $N$ abelian, in which the order-$p$ linear characters inflate from the
$C_p$ quotient and orbit-average away by the same mechanism as in the theorem.
\end{remark}

\begin{remark}
Theorem~\ref{thm:sharp} is a caution for practice: passing to rational or orbit-averaged
character data, a routine normalization, is not arithmetically neutral and privileges
the prime $2$.
\end{remark}

\section{The character landscape: tunability}\label{sec:character}

This section maps the character probe empirically over the $469$ nonabelian groups of
order $\le64$: what layer it tracks, where it fails in both directions, and how it
complements the commuting probe. If the framework were one-trick, every game would read
the same primes. It does not.

\begin{proposition}[No exact order-based law; computed]\label{prop:charmatch}
No single order-based law is exact for the character conductor. Over the $469$ nonabelian
groups of order $\le64$, $\Cchar(G)=\rad(|G|)$ for $446$ ($95.1\%$) and
$\Cchar(G)=\rad(|G/Z(G)|)$ for $418$ ($89.1\%$). The two targets coincide on $p$-groups,
so the gap is carried entirely by the non-$p$-groups, where $\rad(|G|)$ fits better; we
therefore take $\rad(|G|)$ as the probe's tracking target, with $\rad(|G/Z|)$ the
runner-up (indeed $\Cchar(G)\subseteq\rad(|G|)$ unconditionally,
Corollary~\ref{cor:charbound}, so the tracking law can only fail by undershoot). The
guess ``primes dividing character degrees'' is falsified. Exactness on
families is a theorem target.
\end{proposition}

Both counts are order-$\le64$ empirics, dominated by the $311$ nonabelian $2$-groups in
range (out of $469$), on which the two targets coincide at $\{2\}$; the comparison
between the laws is carried by the remaining $158$ groups. Nor are the laws asymptotic:
Theorem~\ref{thm:dihedral} exhibits an infinite family, the dihedral groups $D_{2q}$, on
which $\Cchar=\{2\}$ differs from both targets. That the deviations are structured, the
forced prime of Theorem~\ref{thm:index2}, the silence at $A_5$, the dihedral collapse to
$\{2\}$, is what makes the probe selective rather than a noisy copy of the full-support
law.

\begin{remark}[Case studies]\label{rem:cases}
Three exact data points delimit the character probe. $A_5$ has \emph{empty} character
conductor: its four orbit-trace rows (the two $3$-dimensional characters form one Galois
orbit) span a lattice with unimodular Smith form, so the probe is silent on a perfect group
with $\rad(|G|)=\{2,3,5\}$: both laws of Proposition~\ref{prop:charmatch} fail here by
undershoot. $\mathrm{SL}(2,3)$ has $2\in\Cchar$ although
$|G/G'|=3$ is odd: the converse of Theorem~\ref{thm:index2} fails. And
$C_3\times D_8$ separates the two laws: $\Cchar=\{2,3\}=\rad(|G|)$, strictly exceeding
$\rad(|G/Z|)=\{2\}$, an exact match for the tracking target and an overshoot of the
runner-up.
\end{remark}

\begin{proposition}[Complementarity; computed]\label{prop:compl}
Every one of the $469$ nonabelian groups of order $\le64$ has a nonempty conductor in at
least one probe. The two conductors are disjoint for $20$ of them ($13$ with both
nonempty), and the union recovers the full prime support of the group more often than
either probe alone: $\Ccomm(G)\cup\Cchar(G)=\rad(|G|)$ for $450$ of $469$ ($95.9\%$),
against $446$ for the character conductor (Proposition~\ref{prop:charmatch}) and $342$
for the commuting conductor individually.
\end{proposition}

The margin over the character probe alone is small on this metric, and deliberately so:
recovering $\rad(|G|)$ is a coarse yardstick, not the point of the commuting probe. Its
value is that it reads a different layer (Theorem~\ref{thm:dihedral}), carries structural
theorems (Theorems~\ref{thm:isoclinism} and~\ref{thm:product}), and detects primes outside
$\rad(|G|)$ entirely (Corollary~\ref{cor:nondivisor}), which no probe measured by
$\rad(|G|)$-recovery can register.

Writing $\mathrm{Drop}_{\mathrm{char}}(G)=\rad(|G/Z|)\setminus\Cchar(G)$ for the drops
relative to the runner-up law of Proposition~\ref{prop:charmatch} (the law whose failures
the obstruction below localizes), and
$A_p=[p\mid|\Fit(G)|]$, $B_p=[p\mid|G/\Fit(G)|]$, where $\Fit(G)$ is the Fitting subgroup,
the character drops satisfy a layer
condition that is a necessary obstruction rather than a classifier.

\begin{proposition}[Fitting-layer obstruction; computed]\label{prop:xor}
Over the $616$ pairs $(G,p)$ with $p\mid|G/Z|$, one has
$p\in\mathrm{Drop}_{\mathrm{char}}(G)\Rightarrow A_p\oplus B_p$ for all $31$ drops, with
zero exceptions. The confusion matrix is
$\mathrm{TP}=31,\mathrm{FN}=0,\mathrm{FP}=503,\mathrm{TN}=82$
(precision $\approx0.06$): the XOR is an exact necessary obstruction identifying where
the $\rad(|G/Z|)$ law can fail, not an explanation of the character conductor. The
zero-miss count is not vacuous: under a uniform null model, random placement of the $31$
drops among the $616$ pairs ($534$ of them XOR-positive) avoids all misses with
probability
$\binom{534}{31}/\binom{616}{31}\approx0.012$.
\end{proposition}

\section{Dihedral complementarity and negative results}\label{sec:negatives}

The two arcs meet in a single family: for dihedral groups the two probes read provably
disjoint primes. We prove this, then record the two negative controls that calibrate the
framework.

\begin{theorem}[Dihedral complementarity]\label{thm:dihedral}
Let $q$ be an odd prime. Then $\Cchar(D_{2q})=\{2\}$ and $\Ccomm(D_{2q})=\{q\}$.
\end{theorem}

\begin{proof}
\emph{Character half.} The classes of $D_{2q}$ are the identity, the $(q-1)/2$ rotation
classes $\{r^{\pm k}\}$, and one reflection class. The irreducibles are $\mathbf 1$, the sign
character $\varepsilon$, and $(q-1)/2$ two-dimensional characters
$\chi_j(r^k)=\zeta_q^{jk}+\zeta_q^{-jk}$, $\chi_j(s)=0$, forming a \emph{single} Galois orbit
($(\Z/q)^\times/\{\pm1\}$ acts transitively on $\{1,\dots,(q-1)/2\}$ for $q$ prime). The
orbit-trace table thus has exactly three rows: $\mathbf 1$, $\varepsilon$, and the orbit sum
$\Omega$ with $\Omega(1)=q-1$, $\Omega(r^k)=\sum_{j=1}^{q-1}\zeta_q^{jk}=-1$ for
$k\not\equiv0$, and $\Omega(s)=0$. The difference lattice is spanned by
$v_1=\varepsilon-\mathbf 1$ ($-2$ on the reflection class, $0$ elsewhere) and
$v_2=\Omega-\mathbf 1=(q-2,-2,\dots,-2,-1)$ across (identity, rotations, reflection); these
are $\Q$-independent, so the rational rank is $2$ and
$\Cchar(D_{2q})=\{p:p\mid\Delta_2\}$. A $2\times2$ minor of $[v_1;v_2]$ using two
non-reflection columns vanishes ($v_1$ is zero there); one using a non-reflection column $c$
and the reflection column equals $\pm2\,v_2(c)\in\{\pm4,\pm2(q-2)\}$. As $q-2$ is odd,
$\Delta_2=\gcd\bigl(4,2(q-2)\bigr)=2$, so $\Cchar(D_{2q})=\{2\}$ for every odd prime $q$.
(That $2\in\Cchar$ also follows from Theorem~\ref{thm:index2}; the minor computation
additionally excludes $q$ and every other prime, which Theorem~\ref{thm:sharp} alone does
not.)

\emph{Commuting half.} $D_{2q}$ is a CA-group: the centralizer of a rotation is $C_q$ and of
a reflection is the order-$2$ subgroup it generates, all abelian. The maximal abelian
subgroups are $C_q$ and the $q$ reflection subgroups, so $b=q+1$, with blocks of unequal
sizes $q-1$ and $1$, which Theorem~\ref{thm:CA} permits since it carries no equal-block
hypothesis.
Hence $\Ccomm(D_{2q})=\rad(b-1)=\rad(q)=\{q\}$.
\end{proof}

\emph{Pipeline validation.} The theorem needs no computation. As a check on the
computational instrument rather than on the proof, both halves are reproduced exactly in
GAP for $q=3,5,7,11$ (Table~\ref{tab:dih}), and the character minor computation
symbolically through $q=101$.

\begin{table}[t]
\centering
\begin{tabular}{@{}cccccc@{}}
\toprule
$q$ & $G$ & $|G|$ & $b$ & $\Ccomm$ & $\Cchar$\\
\midrule
3  & $D_6$    & 6  & 4  & $\{3\}$  & $\{2\}$\\
5  & $D_{10}$ & 10 & 6  & $\{5\}$  & $\{2\}$\\
7  & $D_{14}$ & 14 & 8  & $\{7\}$  & $\{2\}$\\
11 & $D_{22}$ & 22 & 12 & $\{11\}$ & $\{2\}$\\
\bottomrule
\end{tabular}
\caption{Verification of Theorem~\ref{thm:dihedral}: $b=q+1$, and the two probes read the
two layers as disjoint primes.}
\label{tab:dih}
\end{table}

This is the cleanest tunability statement: the two probes read the rotation subgroup $C_q$
and the order-$2$ quotient as disjoint primes.

\paragraph{Negative results.}
Two negatives constrain the framework. The \emph{class-algebra game} (structure constants of
the class algebra) has empty conductor for all $469$ nonabelian groups of order $\le64$: a silent probe.
The second control is a \emph{faithful} cyclotomic residue reduction of the character game:
take the full character table $X$ over $K=\Q(\zeta_{\exp G})$, and for a rational prime $p$
and a prime ideal $\mathfrak p$ of $\mathcal O_K$ above $p$, compare the rank of $X$ over
the residue field $\mathcal O_K/\mathfrak p$ with its rank over $K$. Unlike the orbit-trace
probe, this control needs no computation: it is provably full-support. The statement is
classical modular representation theory; we record it with a short self-contained proof.

\begin{proposition}[The faithful reduction is full-support]\label{prop:faithful}
For every finite group $G$, every rational prime $p$, and every prime $\mathfrak p$ of
$\mathcal O_K$ above $p$, the rank of the character table falls modulo $\mathfrak p$ if and
only if $p\mid|G|$. The faithful conductor is therefore exactly $\rad(|G|)$ for every
finite group.
\end{proposition}

\begin{proof}
Write $X^{*}=\overline{X}^{\,\mathsf T}$ for the conjugate transpose. Column orthogonality
gives $X^{*}X=\operatorname{diag}(|C_G(x_i)|)$, so
$\det X\cdot\overline{\det X}=\prod_i|C_G(x_i)|$. Each $\sigma\in\Gal(K/\Q)$ permutes
$\Irr(G)$, hence permutes the rows of $X$, so $\sigma(\det X)=\pm\det X$; taking $\sigma$
to be complex conjugation gives $\overline{\det X}=\pm\det X$, hence
$(\det X)^2=\pm\prod_i|C_G(x_i)|$, and this Galois-fixed algebraic integer lies in $\Z$.
Since $X$ is
nonsingular over $K$, the rank falls modulo $\mathfrak p$ iff $\det X\in\mathfrak p$, iff
$(\det X)^2\in\mathfrak p$, iff $p\mid\prod_i|C_G(x_i)|$; and
$p\mid\prod_i|C_G(x_i)|$ iff $p\mid|G|$, because each $|C_G(x_i)|$ divides $|G|$ and
$|C_G(1)|=|G|$. The criterion is independent of the choice of $\mathfrak p$ above $p$,
ramified or not.
\end{proof}

\begin{remark}[Attribution]\label{rem:brauer}
Both the proposition and more are classical. Brauer theory gives the exact rank of the
reduced table, the number of $p$-regular classes of $G$: one has
$\chi(g)\equiv\chi(g_{p'})\pmod{\mathfrak p}$ for every $\chi$ (the $p$-part of $g$
contributes $p$-power roots of unity, which reduce to $1$), so the $p$-singular columns
collapse onto $p$-regular ones, while on the $p$-regular classes the reductions of the
ordinary characters span the irreducible Brauer characters, which are linearly
independent over the residue field; see \cite[Part~III]{Serre} or \cite{Navarro}. The
determinant identity $(\det X)^2=\pm\prod_i|C_G(x_i)|$ used in the proof is likewise
classical \cite{Isaacs}. We include the short proof only to keep the control
self-contained.
\end{remark}

The same nonsingularity bounds the selective probe unconditionally.

\begin{corollary}[The character probe cannot escape the group order]\label{cor:charbound}
For every finite group $G$, $\Cchar(G)\subseteq\rad(|G|)$.
\end{corollary}

\begin{proof}
Fix $p\nmid|G|$ and a prime $\mathfrak p$ of $\mathcal O_K$ above $p$. The orbit-trace
table is $T=SX$, where $S$ is the $\{0,1\}$ matrix aggregating the rows of $X$ into their
$t$ Galois orbits; the rows of $S$ have disjoint nonempty supports, so $S$ has full row
rank $t$ over every field. By Proposition~\ref{prop:faithful}, $X$ is invertible over
$\mathcal O_K/\mathfrak p$, so $T$ has rank $t$ over $\mathcal O_K/\mathfrak p$, hence
over $\F_p$ ($T$ has rational-integer entries, and the rank of an integer matrix is the
same over every field of characteristic $p$). If
$\sum_{i\ge2}a_i(O_i-O_1)\equiv0\pmod p$ with $O_1,\dots,O_t$ the rows of $T$, then
$\sum_{i\ge2}a_iO_i-\bigl(\sum_{i\ge2}a_i\bigr)O_1\equiv0$, forcing every $a_i\equiv0$;
the difference lattice therefore has full rank $t-1$ modulo $p$, equal to its rational
rank, and $p\notin\Cchar(G)$.
\end{proof}

The faithful reduction is thus a full-support probe that cannot discriminate: it flags
precisely $\rad(|G|)$, always. Corollary~\ref{cor:charbound} turns the containment
observed on all $469$ nonabelian groups of order $\le64$ into a theorem: the orbit-trace
probe forgets layers of the faithful one but can never add primes outside $|G|$, in
proven contrast with the commuting probe, which can (Corollary~\ref{cor:nondivisor}). The
selective probe is the discriminating object, not an artifact.

\section{Open problems}\label{sec:open}

The results above open a program; this section states it. The central open problem is to
classify the Smith torsion of the centralizer-type incidence
matrix $\BG$ (Section~\ref{sec:Bg}) beyond the CA case. On CA-groups the nontrivial Smith
spectrum is $[\,b-1\,]$, so the conductor is
$\rad(b-1)$; the layered case (e.g.\ $\mathrm{SmallGroup}(64,10)$, factors $[2,2,4]$) is open
and is the substance of the program. Direct products are settled at full spectrum level by Theorem~\ref{thm:product}.
Remaining theorem targets: the $p^2$-protection law of Proposition~\ref{prop:p2}, which for
CA-groups is equivalent, by Theorem~\ref{thm:CA}, to the purely arithmetic statement that
$p^2\mid|G'|$ forces $p\mid b-1$, itself open, as is the layered case;
Conjecture~\ref{conj:union}; and exact spectra for the layered families. The
extraspecial groups beyond order $p^3$ (e.g.\ order $32$, spectrum $[2^6,4^4,8]$) are the
first testbed, and by Theorem~\ref{thm:isoclinism} any answer is automatically an isoclinism
statement. The isoclinism question is settled affirmatively by
Theorem~\ref{thm:isoclinism}, which reframes the program: classifying the Smith torsion of
$\BG$ is a classification over isoclinism classes, and the natural refinement is to ask
which isoclinism invariants (beyond $\rad(b-1)$ in the CA case) the spectrum computes. We also ask
whether the diagnostic family $\mathcal P$ of Proposition~\ref{prop:comm} admits a structural
characterization in terms of $\BG$, and whether $\Ccomm(G)$ can be empty for a nonabelian
group; Corollary~\ref{cor:bge3} rules this out for CA-groups.

\section{Reproducibility}\label{sec:repro}

This section is the audit trail: what was computed, on which ranges, with what
instruments, and how every count can be re-derived. All computations are exact and
deterministic. Group theory and integer Smith normal forms are
computed in GAP \cite{GAP} over the SmallGroups library \cite{BEO}; independent cross-checks
are performed in Python; all runs are CPU-only.
Counts in Section~\ref{sec:commuting} are exhaustive to order $128$; the character-probe
and negative-control counts (Sections~\ref{sec:theorem}, \ref{sec:character},
and~\ref{sec:negatives}) to order $64$, the smaller range reflecting the cost of exact
character tables over cyclotomic fields relative to centralizer computations; the CA and $\BG$
verifications are exact on the listed groups. A negative-control suite (Section~\ref{sec:framework})
validates the rank-drop computation against independently derived values before any law
is asserted.
All GAP and Python scripts, together with their run outputs, are available from the author
on request and will be deposited in a public repository, with DOI, with the final
version, so that any claim in the paper can be re-derived from source.
$\mathrm{SmallGroup}(96,64)$, the first-power deviation discussed in
Sections~\ref{sec:commuting} and~\ref{sec:deform}, has conductor
$\{2\}$ while $\rad|G'|=\{2,3\}$, and is recorded as such.

\section*{Disclosure statement}
The author reports there are no competing interests to declare.

\section*{Data availability statement}
All computational claims in this paper are exactly reproducible from the GAP and Python
scripts described in Section~\ref{sec:repro}, which are available, together with their run
outputs, from the author upon reasonable request; a public archive of the scripts and
outputs, with DOI, will accompany the final version.

\section*{Research methodologies involving AI tools}
Large language model assistance was used in the preparation of this manuscript, in
accordance with the Taylor \& Francis policy on research methodologies involving AI tools.
Specifically: (i)~the originality and accuracy of all content has been confirmed by the
author; every theorem was proved and checked by the author, and every computational claim
was verified by the author's own GAP and Python code; (ii)~the AI tools used were Claude
(Anthropic), models Opus~4.8 and Fable~5, used for adversarial review of proof
exposition, suggestions for computational checks, and editorial revision; all
computations reported in this paper were performed and verified by the author's own GAP
and Python scripts; (iii)~the author has checked the
terms of use of the specific AI tool employed and confirms suitability for publication;
(iv)~the author takes full responsibility for the integrity of the whole content, including
the accuracy of all references. The author retains full records of the methods undertaken,
including complete descriptions and records of the prompts used, which will be shared with
the journal upon request.

\end{document}